\documentclass[12pt]{amsart}
\usepackage{a4wide}
\usepackage[hang,flushmargin]{footmisc}

\usepackage{amsmath,amssymb,amsfonts}
\usepackage[foot]{amsaddr}
\usepackage{graphicx}
\usepackage{hyperref}
\usepackage{orcidlink}

\usepackage{mathtools}
\mathtoolsset{showonlyrefs}

\numberwithin{equation}{section}
\newtheorem{theorem}{Theorem}[section]

\newtheorem{proposition}[theorem]{Proposition}

\newtheorem{corollary}[theorem]{Corollary}
\newtheorem{conjecture}[theorem]{Conjecture}

\newcommand{\R}{{\mathbb R}}
\newcommand{\Q}{{\mathbb Q}}
\newcommand{\Z}{{\mathbb Z}}
\newcommand{\N}{{\mathbb N}}
\newcommand{\C}{\mathbb{C}}
\newcommand{\G}{\mathcal{G}}
\newcommand{\F}{\mathcal{F}}
\newcommand{\vol}{{\rm vol}}

\providecommand{\norm}[1]{\lVert#1\rVert}

\author[M.~Faulhuber]{Markus Faulhuber \orcidlink{0000-0002-7576-5724}}
\address{Faculty of Mathematics, University of Vienna \newline Oskar-Morgenstern-Platz 1, 1090 Vienna, Austria}
\email{markus.faulhuber@univie.ac.at}

\begin{document}
\title[Gabor systems with Hermite functions]{Gabor systems with Hermite functions of order \texorpdfstring{$n$}{n} and oversampling greater than \texorpdfstring{$n+1$}{n+1} which are not frames}

\thanks{This research was funded in whole or in part by the Austrian Science Fund (FWF) [\href{https://doi.org/10.55776/P33217}{10.55776/P33217}]. The author thanks Karlheinz Gröchenig for feedback on the initial draft of the paper. The author gratefully acknowledges the comments from the anonymous referees.}

\keywords{Gabor frame, Hermite function, periodic configuration}

\subjclass{42C15; 33C45}

\begin{abstract}
    We show that a sufficient density condition for Gabor systems with Hermite functions over lattices is not sufficient in general. This follows from a result on how zeros of the Zak transform determine the frame property of integer over-sampled Gabor systems.
\end{abstract}

\maketitle
% \vspace*{-1.2cm}

\section{Introduction}
Gabor systems with Hermite functions have some quite satisfying, but just as many mysterious properties. A fundamental result due to Lyubarskii \cite{Lyu92} and Seip and Wallstén \cite{Sei92_1},~\cite{SeiWal92} is that a Gabor system with a Gaussian window is a frame for $L^2(\R)$ if and only if the sampling rate exceeds 1. Indexing the Gaussian as the $0$-th Hermite function, we have the quite satisfying condition, due to Gröchenig and Lyubarskii \cite{GroechenigLyubarskii_Hermite_2007}, \cite{GroechenigLyubarskii_SuperHermite_2009}, that a Gabor system with the $n$-th Hermite function over a lattice $\Lambda \subset \R^2$ is a frame once the density of the lattice is larger than $n+1$, i.e., $\vol(\R^2/\Lambda)<1/(n+1)$. A main difference to the results in \cite{Lyu92}, \cite{Sei92_1}, \cite{SeiWal92} is that the results in \cite{GroechenigLyubarskii_Hermite_2007} and \cite{GroechenigLyubarskii_SuperHermite_2009} rely on the lattice structure.

% Recently, Gabor frames with Hermite functions and oversampling rate $n+1$ were found by Faulhuber, Shafkulovska, and Zlotnikov \cite{FauShaZlo25}. On the other side, simple conditions for the frame property of Gabor systems with Hermite functions are to be ruled out eventually, due to the results of Lemvig~\cite{Lem16} or Horst, Lemvig, and Videbaek \cite{HorLemVid25}.

In this work, we answer the following question, which has been around for about 20~years, since the first announcement of their result by Gröchenig and Lyubarskii:\newline
\textit{Can the sufficiency condition be extended beyond the lattice setting?}
The answer is NO!

This paper provides explicit examples of Gabor systems with the first and third Hermite functions and oversampling rates of 3 and 5, respectively, which are not frames. These are derived from the location of known zeros of their Zak transforms (the existence of further zeros is open). Additionally, we show the existence of a Gabor system with the second Hermite function and oversampling rate 5, and one with the third Hermite function and oversampling rate 7, which are not frames. Our approach is not entirely new: The Zak transform and its zeros have often been used to construct lattices which, together with a (class of) Hermite window(s), do not yield a frame, see, e.g., \cite[Prop.~3.5]{GroechenigLyubarskii_Hermite_2007}, \cite{HorLemVid25}, \cite{Lem16}.
% The attention was, however, always restricted to the lattice setting.
The sets we construct still exhibit periodicities and symmetries, but are not lattices. In \cite{GroRomSto20}, results for Gabor systems with semi-regular lattices $\Gamma \times \Z \subset \R^2$ have been derived. Our examples do not generally show this tensor structure, but we give non-frame examples for the third Hermite function with semi-regular lattices of density 2, 3, and 4.

A related problem concerns the frame set structure for Gabor systems with the $n$-th Hermite functions and lattices of density $1 < D(\Lambda) \leq n+1$ \cite{Gro14}. So far, restrictions for odd Hermite functions and rectangular lattices are known \cite{LyubarskiiNes_Rational_2013}, which have been extended to general lattices \cite{Faulhuber_LyuNes_2019}: The Gabor system with an odd Hermite function is not a frame if the lattice density satisfies $D(\Lambda) = \frac{N+1}{N}$, $N$ a positive integer. For even Hermite windows, there were no further restrictions than the density condition $D(\Lambda) > 1$ until Lemvig provided the first exceptions~\cite{Lem16}. More examples of Gabor systems with Hermite functions and rectangular lattices of density $1 < D(\Lambda) \leq n+1$, followed \cite{HorLemVid25}. For $h_1$, it could still hold that the restriction $D(\Lambda)>\frac{N+1}{N}$, provided in \cite{LyubarskiiNes_Rational_2013}, is the only exception \cite{Gro14}. The results in \cite{FauShaZlo25} broke a pattern occurring for Hermite functions of order up to 3, challenging a question posed by Lemvig~\cite{Lem16}: \textit{Is the lattice density condition of Gröchenig and Lyubarskii sharp}?
This question remains open. If true, it may be answered by knowing the number and location of the zeros of the Zak transforms of generalized Hermite functions.

\medskip

The author thanks the first referee for providing the following remarks and consequences: Our answer may be surprising. At least from a complex analysis viewpoint hints were making it reasonable to believe that a lower Beurling density $D^-(\Gamma) > n+1$ could be sufficient for separated points sets $\Gamma$ to form a Gabor frame with the $n$-th Hermite function (cf.\ \cite{SeiBre93} for the case of multiplicity sampling with Hermite functions). Denoting the upper Beurling density of $\Gamma$ by $D^+(\Gamma)$, it can be shown that a Gabor system with Hermite window is a Riesz basis if $D^+(\Gamma) < 1/(n+1)$. Using the Weierstrass sigma function, this result was established for lattices in \cite[Thm.~1]{AbrGro12}. However, as in \cite{SeiWal92}, the result relies on the growth of the sigma function and can be generalized using a non-uniform version. In the case of a lattice $\Lambda \subset \R^2$, we have $D^+(\Lambda) = D^- (\Lambda) = D(\Lambda)$, which is called the density of the lattice. Duality results (Ron-Shen \cite{RonShe97}, Wexler-Raz \cite{WexRaz90}) yield equivalent conditions between the Riesz basis property and the frame property: A Gabor system is a frame if and only if the dual system with lattice $\Lambda^\circ = D(\Lambda)^{1/2} \Lambda$ is a Riesz basis. Thus, we not only show that the sufficient density result cannot be extended beyond the lattice setting, but our results also prohibit non-uniform versions of the Wexler-Raz or Ron-Shen duality principles. Otherwise, the density result could be deduced from the Riesz basis property.

\medskip

\textbf{Outline of the paper.} We clarify the notation in \S~\ref{sec:notation}: In \S~\ref{sec:Hermite} we introduce our normalization for Hermite functions, which is the same as in the book of Folland~\cite{Fol89}. Time-frequency shifts and multi-window Gabor systems are introduced in \S~\ref{sec:Gabor}, following the book of Gröchenig~\cite{Gro01}. Also, the notion of a periodic structure (the union of finitely many shifted copies of one lattice) is discussed there. We then present the (parameter-free) Zak transform in \S~\ref{sec:Zak}, which maps functions from $\R$ to quasi-periodic functions with period lattice $\Z^2$. This is again covered by the book of Gröchenig~\cite{Gro01}. The reader familiar with the concepts and notation may directly jump to the results in \S~\ref{sec:results}. Finally, in \S~\ref{sec:questions} we pose some open questions on the topic which should be understood as lines of research to advance the status of research. The author thanks the referees for the encouragement to include a discussion on open problems and Hermite windows of low order.

\section{Notation}\label{sec:notation}
In most parts of this article, we will follow the notation from the book of Gröchenig~\cite{Gro01}. The Hilbert space of square-integrable functions on the real line is denoted by $L^2(\R)$. For technical reasons, it will sometimes be necessary to restrict ourselves to a suitable dense subspace, in which case we will also write that a function is suitable. The statement is then valid for the Hilbert space by the density of the suitable function space, considering that some results will only hold almost everywhere and not pointwise.

We say that a function is suitable if it is in the modulation space $M^1(\R)$, also known as the Feichtinger algebra $S_0(\R)$ \cite{Fei83}. In some situations, it would suffice to pick a function from the (slightly larger) Wiener space of continuous functions $W_0(\R)$ with Fourier transform also contained in $W_0(\R)$, but to avoid confusion, we only consider $M^1(\R)$. As our results will be for Hermite functions, which are contained in the Schwartz space $\mathcal{S}(\R)$, the reader not familiar with $M^1(\R)$ may consider functions in $\mathcal{S}(\R) \subset M^1(\R)$ as suitable.

The Fourier transform of a suitable function on $\R$ is given by
\begin{equation}
    \F f(y) = \widehat{f}(y) = \int_{\R} f(t) e^{-2 \pi i y t} \, dt.
\end{equation}
Note that $\F$ extends to a unitary operator on $L^2(\R)$, i.e.,
\begin{equation}
    \langle f, g \rangle = \langle \F f, \F g \rangle,
    \quad \text{ where } \quad
    \langle f, g \rangle = \int_{\R} f(t) \overline{g(t)} \, dt.
\end{equation}
Moreover, we will use the unitary dilation operator $\mathcal{D}_a$, $a > 0$, which acts by the rule
\begin{equation}\label{eq:dilation}
    \mathcal{D}_a f(t) = \frac{1}{\sqrt{a}} f \left( \frac{t}{a} \right).
\end{equation}

\subsection{Hermite functions}\label{sec:Hermite}
We consider Hermite functions as defined in \cite{Fol89}. The $n$-th order Hermite function is the $n$-th Hermite polynomial $H_n$ (subject to our normalization) multiplied by a Gaussian function $\phi(t) = e^{-\pi t^2}$. More precisely, we have
\begin{equation}
    h_n(t) = (-1)^n C_n e^{\pi t^2} \dfrac{d^n}{dt^n}\left( e^{-2 \pi t^2} \right) = H_n(t) \phi(t),
\end{equation}
where $C_n$ is a normalizing constant, which is explicitly given by $C_n = 2^{1/4}/\sqrt{n! (2\pi)^n 2^n}$. In particular, for $n=0$ we obtain a normalized Gaussian function
\begin{equation}
    h_0(t) = C_0 \, \phi(t) = 2^{1/4} e^{-\pi t^2}.
\end{equation}
Moreover, the Hermite functions are eigenfunctions of the Fourier transform with
\begin{equation}
    \F h_n = (-i)^n h_n, \quad n \in \N_0.
\end{equation}

\subsection{Gabor systems}\label{sec:Gabor}
A Gabor system in $L^2(\R)$ is a function system of the type
\begin{equation}
    \G(g,\Gamma) = \{\pi(\gamma) g \mid \gamma \in \Gamma \subset \R^2\},
\end{equation}
where $\pi(\gamma)$ is a time-frequency shift. This is the composition of a translation (time-shift) and a modulation (frequency-shift), i.e.,
\begin{equation}
    \pi(z) = M_\omega T_x, \ z=(x,\omega)
    \quad
    T_x f(t) = f(t-x),
    \quad
    M_\omega f(t) = e^{2 \pi i \omega t} f(t).
\end{equation}
Time-frequency shifts are not closed under composition due to the commutation relation
\begin{equation}\label{eq:CR_TF-shifts}
    M_\omega T_x = e^{2 \pi i \omega x} T_x M_\omega.
\end{equation}
However, by adding a phase factor $c \in \C$, $|c|=1$, they become a non-commutative group:
\begin{equation}
    e^{2 \pi i \tau_1} \, \pi(z_1) \, e^{2 \pi i \tau_2} \, \pi(z_2) = e^{2 \pi i (\tau_1+\tau_2-x_1\omega_2)} \, \pi(z_1+z_2), \quad z_1=(x_1,\omega_1), \, z_2=(x_2,\omega_2).
\end{equation}
For more details, refer to \cite[Chap.~1.3]{Fol89}, \cite[Chap.~9.1]{Gro01}. The important part for us here is that the composition of two time-frequency shifts is (up to a phase factor) again a time-frequency shift by the sum of the coordinates.

A Gabor system is a frame for $L^2(\R)$ if and only if there exist constants $0 < A \leq B < \infty$, called frame bounds, such that the following frame inequality is satisfied:
\begin{equation}\label{eq:frame}
    A \norm{f}_{L^2}^2 \leq \sum_{\gamma \in \Gamma} | \langle f, \pi(\gamma) g \rangle|^2 \leq B \norm{f}_{L^2}^2, \quad \forall f \in L^2(\R).
\end{equation}
The optimal frame bounds are the spectral bounds of the associated Gabor frame operator
\begin{equation}
    S_{\G(g,\Gamma)} f = \sum_{\gamma \in \Gamma} \langle f, \pi(\gamma) g \rangle \, \pi(\gamma) g.
\end{equation}
We remark that we always assume $\Gamma$ to be separated (uniformly discrete), i.e.,
\begin{equation}
    \inf_{\substack{\gamma_1, \gamma_2 \in \Gamma \\ \gamma_1 \neq \gamma_2}}(|\gamma_1-\gamma_2|) \geq \delta > 0.
\end{equation}
In this case, for $g \in M^1(\R)$, the upper frame bound is always finite, and we only need to care about the strict positivity of the lower frame bound. In the sequel, we will consider periodic configurations. The points then, indeed, have a minimal distance from each other.

A multi-window Gabor system is of the form
\begin{equation}
    \mathfrak{G}(\{g_m\}_{m=1}^M,\Gamma) = \bigcup_{m=1}^M \G(g_m,\Gamma)
\end{equation}
and the associated multi-window frame operator is given by
\begin{equation}
    S_{\mathfrak{G}(\{g_m\}_{m=1}^M,\Gamma)} = \sum_{m=1}^M S_{\G(g_m,\Gamma)} = \sum_{m=1}^M \sum_{\gamma \in \Gamma} \langle f, \pi(\gamma) g_m \rangle \, \pi(\gamma) g_m.
\end{equation}
The windows $g_m$ may be completely different, but we are interested in the case where they are shifted copies $\pi(z_m) g$ of a single function $g$. A direct computation shows that the multi-window frame operator is also the associated frame operator of a Gabor system over a periodic structure $\Gamma(\Lambda,\{z_m\}_{m=1}^M)$, $\Lambda$ a lattice, with window $g$. The periodic structure is
\begin{equation}
    \Gamma(\Lambda,\{z_m\}_{m=1}^M) = \bigcup \limits_{m=1}^M (\Lambda + z_m)
\end{equation}
To simplify notation, we set $\Gamma_M(\Lambda) = \Gamma(\Lambda,\{z_m\}_{m=1}^M)$, keeping in mind that we still need to choose the set $\{z_1, \ldots, z_M\}$. While $\Gamma_M(\Lambda)$ is still $\Lambda$-periodic, it is not necessarily a lattice. The density, or (over-)sampling rate, of a periodic structure is given by
\begin{equation}
    D(\Gamma_M(\Lambda)) = M \cdot D(\Lambda) = \frac{M}{\vol(\R^2/\Lambda)}.
\end{equation}
Assume $g$ is suitable (so the following series converges unconditionally), then
\begin{align}\label{eq:frame_shifts}
    S_{\mathfrak{G}(\{\pi(z_m) g\}_{m=1}^M,\Lambda)} f
    & = \sum_{m=1}^M \sum_{\lambda \in \Lambda} \langle f, \pi(\lambda) \pi(z_m) g \rangle \, \pi(\lambda) \pi(z_m) g
    \\
    & = \sum_{m=1}^M \sum_{\lambda \in \Lambda} \langle f, \pi(\lambda+z_m) g \rangle \, \pi(\lambda+z_m) g
    = S_{\G(g, \Gamma_M)} f.
\end{align}
The phase factor, which appears due to \eqref{eq:CR_TF-shifts}, also appears as a complex conjugate in the above calculation and cancels. The last detail we collect is the unitary equivalence of the Gabor systems $\G(g,\Gamma)$ and $\G(\mathcal{D}_a g, D_a \Gamma)$ (see \cite{Fau25_SampTA}, \cite{Gos15}), where
\begin{equation}
    D_a = \begin{pmatrix}
        a & 0 \\
        0 & \frac{1}{a}
    \end{pmatrix}, \quad a > 0.
\end{equation}
First, for $z \in \R^2$, we note that $\mathcal{D}_a \,\pi(z) \, \mathcal{D}_a^{-1} = \pi(D_a z)$, which can be checked by a simple calculation with a suitable function. From this, we get
\begin{align}
    S_{\G(\mathcal{D}_a g, D_a \Gamma)} f
    & = \sum_{\gamma \in D_a \Gamma} \langle f , \pi(\gamma) \mathcal{D}_a g \rangle \, \pi(\gamma) \mathcal{D}_a g
    = \sum_{\gamma \in \Gamma}  \langle f , \pi(D_a \gamma) \mathcal{D}_a g \rangle \, \pi(D_a \gamma) \mathcal{D}_a g\\
    & = \sum_{\gamma \in \Gamma}  \langle f , \mathcal{D}_a \pi(\gamma) \mathcal{D}_a^{-1} \mathcal{D}_a g \rangle \, \mathcal{D}_a \pi(\gamma) \mathcal{D}_a ^{-1} \mathcal{D}_a g
    = \mathcal{D}_a \sum_{\gamma \in \Gamma}  \langle \mathcal{D}_a^{-1} f , \pi(\gamma) g \rangle \, \pi(\gamma) g\\
    & = \mathcal{D}_a S_{\G(g,\Gamma)} \mathcal{D}_a^{-1} f.
\end{align}
In the above computations, we used the unconditional convergence of the series (because $g$ is suitable, so we have a Bessel system, i.e., a finite upper frame bound) and the fact that $\mathcal{D}_a$ is unitary. In particular, the two Gabor systems have the same optimal frame bounds.

\subsection{The Zak transform}\label{sec:Zak}
The Zak transform has become a standard tool in time-frequency analysis to study the frame property of integer over-sampled Gabor systems. For a suitable function $f$, its Zak transform is given by 
\begin{equation}
    \mathcal{Z} f(x,\omega) = \sum_{k \in \Z} f(k-x) e^{2 \pi i k \omega} = \sum_{k \in \Z} M_\omega T_x f(k).
\end{equation}
In shorter notation, setting $z=(x,\omega)$ we can also write this as
\begin{equation}
    \mathcal{Z}f(z) = \sum_{k \in \Z} \pi(z) f(k).
\end{equation}
We note the following basic properties, which are easily verified.
\begin{equation}\label{eq:Zak_periodic}
    \mathcal{Z}f(x+1,\omega) = e^{2 \pi i \omega} \mathcal{Z}f(x,\omega)
    \quad \text{ and } \quad
    \mathcal{Z}f(x,\omega+1) = \mathcal{Z}f(x,\omega)
\end{equation}
Thus, $\mathcal{Z}f(x,\omega)$ is quasi-periodic with respect to the integer lattice $\Z^2$ and completely determined by the values in $[0,1)^2$. Next, we note that a time-frequency shift of the function basically results in a shift of the Zak transform. Due to \eqref{eq:CR_TF-shifts}, we obtain
\begin{equation}\label{eq:Zak_TFshift}
    \mathcal{Z} (M_\eta T_\xi f)(x,\omega) = e^{-2 \pi i \eta x} \mathcal{Z} f(x+\xi,\omega+\eta).
\end{equation}
If we consider the Zak transform of a Fourier transform of a function, then we obtain
\begin{equation}\label{eq:FT-Zak}
    \mathcal{Z} \widehat{f}(x,\omega) = e^{2 \pi i x \omega} \mathcal{Z} f(\omega, -x).
\end{equation}
This follows directly from the Poisson summation formula, which, in its standard form, is
\begin{equation}
    \sum_{k \in \Z} f(k) = \sum_{l \in \Z} \widehat{f}(l).
\end{equation}
As the Fourier transform intertwines translation and modulation, precisely
\begin{equation}
    \F T_x = M_{-x} \F
    \quad \text{ and } \quad
    \F M_\omega = T_\omega \F,
\end{equation}
the phase factor in \eqref{eq:Zak_TFshift} appears then due to the commutation relation \eqref{eq:CR_TF-shifts}. We also note that the Zak transform extends to a unitary operator from $L^2(\R)$ onto $L^2([0,1)^2)$. We refer to \cite[Thm.~8.2.3]{Gro01} for the details of this fact.

We close the survey on the Zak transform by listing what we call the trivial zeros. They occur due to the parity of $f$. For the details see \cite[\S~5]{Jan88} or \cite[\S~4]{HorLemVid25}. We have (see Figure~\ref{fig:zeros}):
\begin{itemize}
    \item For a suitable $f$, if $f(t) = f(-t)$, so $f$ is an even function, then
    \begin{equation}\label{eq:Zak_zeros_even}
        \mathcal{Z} f\left(\tfrac{1}{2}+k, \tfrac{1}{2}+l \right) = 0, \quad (k,l) \in \Z^2,
    \end{equation}
    \item For a suitable $f$, if $f(t) = -f(-t)$, so $f$ is an odd function, then
    \begin{equation}\label{eq:Zak_zeros_odd}
        \mathcal{Z} f\left(k, l \right) = \mathcal{Z} f\left(\tfrac{1}{2}+k, l \right) = \mathcal{Z} f\left(k, \tfrac{1}{2}+l \right) = 0, \quad (k,l) \in \Z^2.
    \end{equation}
\end{itemize}
\begin{figure}[ht]
    \centering
    \includegraphics[width=0.425\linewidth]{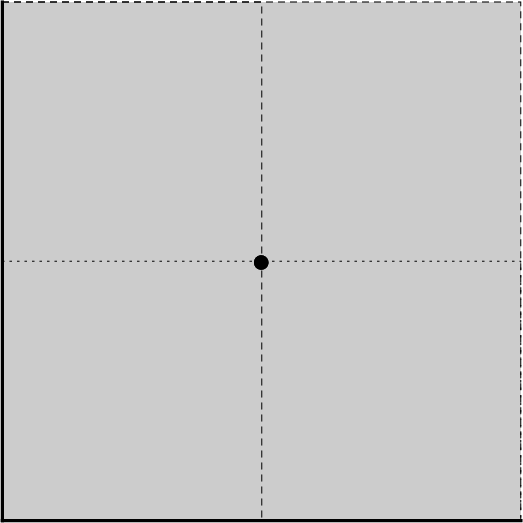}
    \hfill
    \includegraphics[width=0.425\linewidth]{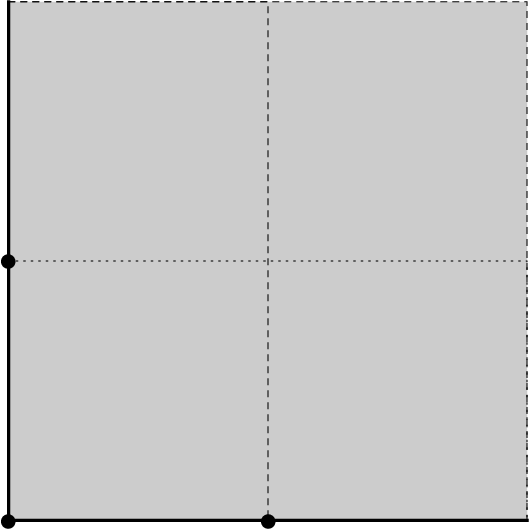}
    \caption{Zeros of $\mathcal{Z} f(x,\omega)$ of a suitable even (left) and odd (right) function in $[0,1)^2$, obtained by parity. Further zeros may exist, depending on $f$.}
    \label{fig:zeros}
\end{figure}

Here, the condition that $f$ is suitable can be weakened to $f \in W_0(\R)$ but cannot be extended to $f \in L^2(\R)$. It implies that $\mathcal{Z} f$ is continuous (see \cite[Prop.~8.2.1~(c)]{Gro01}). A remarkable property of the Zak transform is that, if $\mathcal{Z} f$ is continuous, then it has (at least) a zero in $[0,1)^2$. The indicator function $\chi_{[0,1]}$ is not suitable and not in $W_0(\R)$. It is not continuous, hence, not in $M^1(\R) \subset W_0(\R) \cap \F W_0(\R)$, and $|\mathcal{Z} \chi_{[0,1]}| = 1$ almost everywhere.

\section{Results}\label{sec:results}
We now present a result of how the zeros of $\mathcal{Z} g$ determine the frame property of integer oversampled Gabor systems. The new aspect is that we do not insist that the underlying set is a lattice. Thus, the result belongs to the study of non-uniform Gabor systems. We refer to \cite{AscFeiKai14}, \cite{DoeRom14}, \cite{FeiKai04}, \cite{GroOrtRom15} to name just a few results in this direction.
\begin{theorem}\label{thm:frame_zeros}
    For $g \in M^1(\R)$, let $\{z_1, \ldots, z_M\} \subset [0,1)^2$ be distinct zeros of its Zak transform, i.e., $\mathcal{Z} g(z_m) = 0$, $m = 1, \ldots , M$. Then, the multi-window Gabor system $\mathfrak{G}(\{\pi(z_m)g\}_{m=1}^M,\Z^2)$ is not a frame for $L^2(\R)$. Equivalently, we have that the Gabor system $\G \left(g, \bigcup_{m=1}^M (\Z^2+ z_m) \right)$ is not a frame for $L^2(\R)$.
\end{theorem}
\begin{proof}
    To keep this article self-contained, we repeat a standard proof, as given in the book of Gröchenig \cite[Chap.~8.3]{Gro01}, that the Zak transform diagonalizes the frame operator. More precisely, for a multi-window Gabor system we will show (cf.~\cite[eq.~(8.18)]{Gro01}) that
    \begin{equation}
        \mathcal{Z} \sum_{m=1}^M S_{\G(g_m,\Z^2)} f = \left(\sum_{m=1}^M |\mathcal{Z} g_m|^2\right) \mathcal{Z} f, \quad f \in L^2(\R).
    \end{equation}
    To keep notation simple, we carry out the calculations for the case of a single window, as
    \begin{equation}
        \mathcal{Z} \left(\sum_{m=1}^M S_{\G(g_m,\Z^2)} f \right) = \sum_{m=1}^M \left(\mathcal{Z} S_{\G(g_m,\Z^2)} f \right).
    \end{equation}
    We start with combining \eqref{eq:Zak_TFshift} and \eqref{eq:Zak_periodic} to conclude that
    \begin{equation}\label{eq:Zak-TF-shift_kl}
        \mathcal{Z}(M_l T_k g_m)(x,\omega) = e^{-2 \pi i l x} \mathcal{Z} g_m (x+k,\omega+l) = e^{2 \pi i (k \omega - l x)} \mathcal{Z} g_m(x,\omega), \quad k,l \in \Z.
    \end{equation}
    Next, using that $\mathcal{Z}: L^2(\R) \to L^2([0,1]^2)$ is unitary, we compute
    \begin{align}
        \mathcal{Z} (S_{\G(g_m,\Z^2)} f)(x,\omega)
        & = \sum_{(k,l)\in \Z^2} \langle f, M_l T_k g_m \rangle_{L^2(\R)} \, \mathcal{Z} (M_l T_k g_m)(x,\omega)\\
        & = \left(\sum_{(k,l) \in \Z^2} \langle \mathcal{Z} f, \mathcal{Z} (M_l T_k g_m) \rangle_{L^2([0,1]^2)} \, e^{2\pi i (k \omega - l x)} \right) \, \mathcal{Z} g_m (x, \omega).\label{eq:ZSZ-1f}
    \end{align}
    We have a look at the inner product. Using \eqref{eq:Zak-TF-shift_kl}, we obtain
    \begin{equation}
        \langle \mathcal{Z} f, \mathcal{Z} (M_l T_k g_m) \rangle_{L^2([0,1]^2)} = \iint_{[0,1]^2} \mathcal{Z}f(x,\omega) \overline{\mathcal{Z} g_m (x,\omega)} e^{-2 \pi i (k \omega - l x)} \, d(x,\omega).
    \end{equation}
    These are Fourier coefficients of $\mathcal{Z} f \overline{\mathcal{Z}g_m}$, and the whole expression in the parentheses in \eqref{eq:ZSZ-1f} is simply a Fourier series of this expression. Thus,
    \begin{equation}
        \sum_{m=1}^M \mathcal{Z} S_{\G(g_m,\Z^2)} f = \left(\sum_{m=1}^M |\mathcal{Z} g_m|^2 \right) \mathcal{Z} f
        \quad \Longleftrightarrow \quad
        \sum_{m=1}^M \mathcal{Z} S_{\G(g_m,\Z^2)} \mathcal{Z}^{-1} = \sum_{m=1}^M |\mathcal{Z} g_m|^2.
    \end{equation}
    This was the first part of the proof. We also see that $\sum_{m=1}^M\mathcal{Z} S_{\G(g,\Z^2)} \mathcal{Z}^{-1}$ is a multiplication operator which has the same spectrum as the frame operator. A multiplication operator $f \mapsto M f$ is bounded if and only if $M \in L^\infty$ and invertible if and only if $M^{-1} \in L^\infty$. Hence,
    \begin{equation}
        \bigcup_{m=1}^M \G\left(\pi(z_m) g, \Z^2\right) \text{ is a frame}
        \quad \Longleftrightarrow \quad
        0 < A \leq \sum_{m=1}^M |\mathcal{Z}(\pi(z_m) g)(x,\omega)|^2 \leq B < \infty.
    \end{equation}
    By using \eqref{eq:Zak_TFshift} again, we see that
    \begin{equation}
        \sum_{m=1}^M |\mathcal{Z}(\pi(z_m) g)(z)|^2 = \sum_{m=1}^M |\mathcal{Z}(g)(z+z_m)|^2, \quad z=(x,\omega).
    \end{equation}
    By assumption $\mathcal{Z}g(z_m) = 0$ for $m=1, \ldots , M$. Thus, by choosing $z=0$, we see that
    \begin{equation}\label{eq:frame_zeros}
        \min_{z \in [0,1)^2} \sum_{m=1}^M|\mathcal{Z}(g)(z+z_m)|^2 = \sum_{m=1}^M|\mathcal{Z}(g)(z_m)|^2 = 0.
    \end{equation}
    forcing the lower frame bound to vanish, and the result is proved.
\end{proof}
The methods in the proof of Theorem~\ref{thm:frame_zeros} may be considered standard in time-frequency analysis. The novel aspect is that we apply it to a periodic structure, which can be decomposed into finitely many shifted copies of a lattice. So far, the method has only been applied to lattices that could be separated into shifts of a coarser lattice. In the same spirit, our approach can be extended to periodic structures of rational density by using techniques as presented by Zibulski and Zevii \cite{ZibZee97} (see also \cite[Chap.~8.3]{Gro01} or \cite{LyubarskiiNes_Rational_2013}).

Note that the result also holds for $\mathcal{G}(g, \bigcup_{m=1}^M(\Z^2+z_m+\xi))$, i.e., if we consider a global shift by $\xi \in \R^2$, as the location of all zeros in \eqref{eq:frame_zeros} is shifted simultaneously. As $\pi$ is unitary, the invariance under global shifts generally reflects itself in the frame inequality~\eqref{eq:frame}:
\begin{equation}
    \sum_{\gamma \in \Gamma+\xi} |\langle f, \pi(\gamma) g \rangle|^2 = \sum_{\gamma \in \Gamma} |\langle f, \pi(\gamma+\xi) g \rangle|^2 = \sum_{\gamma \in \Gamma} |\langle \pi^{-1}(\xi)f, \pi(\gamma) g \rangle|^2, \quad \forall f \in L^2(\R).
\end{equation}
With these arguments, all results for Gabor systems hold true for global shifts in the sequel.

At this point, we remark that several versions of the Zak transform exist. For example, in \cite{HorLemVid25} or \cite{Lem16} the following version is used (re-call $\mathcal{D}_a$ from \eqref{eq:dilation})
\begin{equation}\label{eq:tilde_Zak_a}
    \widetilde{\mathcal{Z}}_a f(x,\omega) = \sqrt{a} \sum_{k \in \Z} f(a(k-x)) e^{2 \pi i k \omega} = \sum_{k \in \Z} M_\omega T_x \mathcal{D}_a^{-1} f(k) = \mathcal{Z} (\mathcal{D}_a^{-1} f)(x,\omega).
\end{equation}
This can be used to study Gabor systems of the form $\G(\mathcal{D}_a^{-1} g, \Z^2)$. By unitary equivalence, this is the same as studying properties of $\G(g, D_a \Z^2)$, so a Gabor system with window g over the rectangular lattice $a \Z \times \frac{1}{a} \Z$. On the other hand, there is the following version, which is, for example, presented in \cite{Jan88}. Setting $z = (x,\omega) \in \R^2$, we can write this version~as
\begin{align}\label{eq:Zak_a}
    \mathcal{Z}_a f(z)
    & = \sqrt{a} \sum_{k \in \Z} f(a k - x) e^{2 \pi i a k \omega}
    = \sum_{k \in \Z} \mathcal{D}_a^{-1} \pi(z) f(k)
    = \sum_{k \in \Z} \pi(D_a^{-1} z) \mathcal{D}_a^{-1} f(k)\\
    & = \mathcal{Z} (\mathcal{D}_a^{-1}f) (D_a^{-1} z).
\end{align}
The essential properties from $\mathcal{Z}$ carry over to $\mathcal{Z}_a$, but the main difference is that $\mathcal{Z}_a$ is quasi-periodic with respect to the lattice $a \Z \times \frac{1}{a} \Z$. Hence, we have that $\mathcal{Z}_a$ is determined by its values on $[0,a) \times [0,\frac{1}{a})$ and it is a unitary operator from $L^2(\R)$ onto $L^2(D_a[0,1]^2)$. This version of the Zak transform, combined with the techniques presented in the proof above, has long been used to study Gabor systems over rectangular lattices of the form $a \Z \times b \Z$ with $(ab)^{-1} \in \N$. All results in this article carry over to $\mathcal{Z}_a$ and $a \Z \times \frac{1}{a} \Z$ with minor adjustments. This technique has also been expanded to the case that $(ab)^{-1} = \frac{p}{q} \in \Q$, leading to determining whether certain $q \times q$ matrices have full rank (see \cite[Chap.~8.3]{Gro01}). A rectangular lattice as an index set is quite restrictive, as the shifts $\pi(\xi_n)$ must possess very restricted symmetries. This is discussed to some extent in \cite[\S~3]{HorLemVid25}.

The freedom of allowing arbitrary periodic configurations leads to the following result.
\begin{proposition}\label{pro:frame_zeros}
    Let $g \in M^1(\R)$ and assume $\mathcal{Z}g$ only has $N$ separated zeros in $[0,1)^2$. Pick $K > 1$ points $\{\xi_1, \ldots , \xi_K\}$ from $[0,1)^2$ i.i.d.\ randomly, then the Gabor system
    \begin{equation}
        \G\left(g, \bigcup_{k=1}^K (\Z^2+\xi_k) \right) \text{ is a frame with probability 1.}
    \end{equation}
    If $K > N$, then $\G(g, \bigcup_{k=1}^K (\Z^2+\xi_k))$ is a frame for any choice of $K$ distinct points in $[0,1)^2$.
\end{proposition}
\begin{proof}
    The fact that $g \in M^1(\R)$ guarantees that we have a finite upper frame bound. Thus, similar to \eqref{eq:frame_zeros}, in order to have a frame, we need to ensure that
    \begin{equation}
        \min_{z \in [0,1)^2} \sum_{m=1}^M |\mathcal{Z} g (z+z_m)|^2 > 0.
    \end{equation}
    As the zero set of $\mathcal{Z} g$ is discrete and finite by assumption, the probability that $K$ i.i.d.\ randomly picked points from $[0,1)^2$ all belong to (a shift of) the zero set of $\mathcal{Z}g$ is zero.
\end{proof}

A consequence of Theorem~\ref{thm:frame_zeros} and the results of Gröchenig and Lysubarskii \cite{GroechenigLyubarskii_Hermite_2007},~\cite{GroechenigLyubarskii_SuperHermite_2009} concerns the structure of the location of the zeros of Zak transforms of Hermite functions.
\begin{corollary}\label{cor:location_zeros}
    For the $n$-th Hermite function $h_n$, the Zak transform $\mathcal{Z} h_n$ can have at most $n+1$ zeros $\{z_1, \ldots , z_{n+1}\} \subset [0,1)^2$ such that
    \begin{equation}
        \Gamma(\Z^2,\{z_\ell\}_{\ell=1}^{n+1}) = \bigcup_{\ell = 1}^{n+1} (\Z^2 + z_\ell) \text{ is a (shifted) lattice.}
    \end{equation}
\end{corollary}
\begin{proof}
    Due to the results of Gröchenig and Lyubarskii \cite{GroechenigLyubarskii_Hermite_2007},~\cite{GroechenigLyubarskii_SuperHermite_2009} we know that a Gabor system $\G(h_n, \Lambda)$ over a lattice $\Lambda$ with $D(\Lambda) > (n+1)$ is a frame. If $\{z_1, \ldots , z_M\}$, $M > n+1$, are zeros of $\mathcal{Z}h_n$ and $\Gamma = \Gamma(\Z^2,\{z_m\}_{m=1}^M)$, then $\Gamma$ has density $D(\Gamma) = M > n+1$. As $\mathcal{G}(h_n,\Gamma)$ is not a frame by Theorem~\ref{thm:frame_zeros}, $\Gamma$ cannot be a lattice.
\end{proof}
Note that this does not exclude the option that $\mathcal{Z}h_n$ has more than $n+1$ zeros in $[0,1)^2$. This is a key message of the paper, which leads to the following observation:

\medskip

    \textit{There exist Gabor systems with Hermite functions of order $n \geq 1$ over periodic index sets of density strictly greater than $n+1$ which are not frames.}

\medskip

Although the following result is an easy consequence of Theorem~\ref{thm:frame_zeros}, it marks the starting point of the study of Gabor systems $\G(h_n,\Gamma)$ with Hermite functions, $n\geq 1$, where $\Gamma$ is not a lattice but a (more general) periodic set.
\begin{theorem}\label{thm:h1}
    Set $z_1 = (0,0)$, $z_2=(1/2,0)$, and $z_3=(0,1/2)$. Then, the periodic configuration $\Gamma(\Z^2,\{z_m\}_{m=1}^3)$ has density 3 and the Gabor system $\G\left(h_1,\Gamma(\Z^2,\{z_m\}_{m=1}^3)\right)$ is not a Gabor frame for $L^2(\R)$.
\end{theorem}
\begin{proof}
    Since $h_1(t) = - h_1(-t)$ is an odd function, we know that
    \begin{equation}\label{eq:zeros_Zh1}
        \mathcal{Z} h_1(0,0) = \mathcal{Z} h_1\left(\tfrac{1}{2},0\right) = \mathcal{Z} h_1\left(0,\tfrac{1}{2}\right) = 0.
    \end{equation}
    The result now follows from Theorem~\ref{thm:frame_zeros}.
\end{proof}
This provides the negative answer as announced in the introduction. We exploited the fact that the Zak transform $\mathcal{Z}h_1$ has (at least) three zeros in $[0,1)^2$. Thus, by Theorem~\ref{thm:frame_zeros}, there exists a periodic configuration $\Gamma$ of density 3 such that $\G( h_1, \Gamma)$ is not a frame. In addition, we know the exact location of the three zeros making $\Gamma$ explicit, which we depict in Figure~\ref{fig:union3}. While it is not a lattice, it still exhibits periodicities and symmetries.
\begin{figure}[ht]
    \centering
    \includegraphics[width=0.75\linewidth]{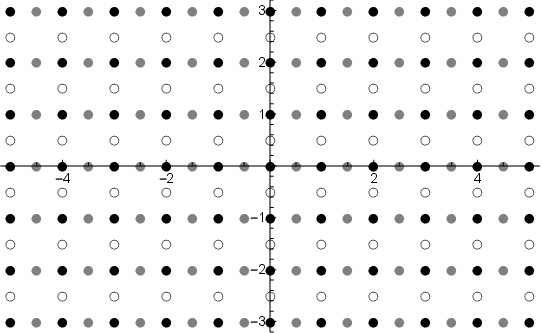}
    \caption{The set $\Gamma = \bigcup_{m=1}^3 (\Z^2 +z_m)$, $z_m \in \{ (0,0), (\frac{1}{2},0), (0,\frac{1}{2})\}$ and its decomposition into three relatively shifted copies of the integer lattice $\Z^2$.}
    \label{fig:union3}
\end{figure}
    
\medskip

\textbf{Examples}. We now give further examples of Gabor systems which are not frames and where the density exceeds the order of the Hermite function by more than 1. Consider the third Hermite function $h_3$. In \cite{Lem16}, Lemvig lists zeros of $\mathcal{Z} \mathcal{D}_{\sqrt{a}}^{-1}h_{4\ell+3}$, $\ell \in \N_0$, $a \in \{\frac{1}{4},\frac{1}{3},\frac{1}{2},2,3,4\}$. Boon, Zak, and Zucker earlier discovered these in a different context for $\mathcal{Z}_{\sqrt{a}} h_{4\ell+3}$ \cite{BooZakZuc83}. Due to a different normalization of the Fourier transform, $\mathcal{Z}_a$ is defined slightly differently in \cite{BooZakZuc83} compared to \eqref{eq:Zak_a}. We have the following known zeros (further zeros may exist), as listed in \cite{BooZakZuc83} (subject to our normalization):

$\bullet$ $\mathcal{Z}_{\sqrt{2}} h_{4\ell+3} (z) = 0$, $z \in \left \lbrace(0,0), \left(\frac{\sqrt{2}}{2},0 \right), \left(0, \frac{1}{2 \sqrt{2}}\right), \left(\frac{\sqrt{2}}{4}, \frac{1}{2 \sqrt{2}}\right), \left(\frac{3\sqrt{2}}{4}, \frac{1}{2 \sqrt{2}}\right) \right \rbrace + D_{\sqrt{2}}\Z^2$.

\noindent
The location of the zeros in $[0,\sqrt{2}) \times [0,\frac{1}{\sqrt{2}})$ is depicted in Figure~\ref{fig:zeros2}.
\begin{figure}[ht]
    \centering
    \includegraphics[width=0.75\linewidth]{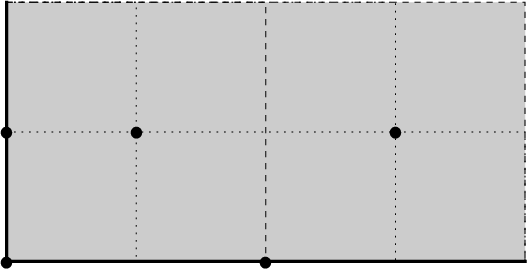}
    \caption{Known zeros of $\mathcal{Z}_{\sqrt{2}} h_{4\ell+3}$ in the fundamental cell of $\sqrt{2} \Z \times \frac{1}{\sqrt{2}} \Z$.}
    \label{fig:zeros2}
\end{figure}

\medskip

$\bullet$ $\mathcal{Z}_{\sqrt{3}} h_{4\ell+3} (z) = 0$, $z \in \left \lbrace(0,0), \left(0, \frac{1}{2 \sqrt{3}}\right), \left(\frac{\sqrt{3}}{2},0 \right), \left(\frac{\sqrt{3}}{3}, 0 \right), \left(\frac{2\sqrt{3}}{3}, 0\right) \right \rbrace + D_{\sqrt{3}}\Z^2$.

\noindent
The location of the zeros in $[0,\sqrt{3}) \times [0,\frac{1}{\sqrt{3}})$ is depicted in Figure~\ref{fig:zeros3}.
\begin{figure}[ht]
    \centering
    \includegraphics[width=0.75\linewidth]{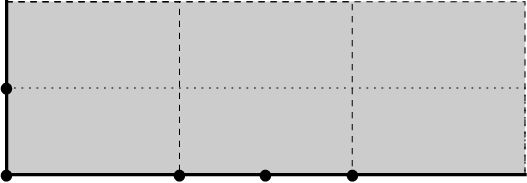}
    \caption{Known zeros for $\mathcal{Z}_{\sqrt{3}} h_{4\ell+3}$ in the fundamental cell of $\sqrt{3} \Z \times \frac{1}{\sqrt{3}} \Z$.}
    \label{fig:zeros3}
\end{figure}

Note that, after dilation, the four zeros with $\omega = 0$ can be used to construct examples of a semi-regular lattice $\Gamma_N \times \Z$ of density $N \in \{2, 3, 4\}$ such that the Gabor system with $\mathcal{D}_{\sqrt{3}}^{-1} h_3$ is not a frame. Denote by $z_1 = (0,0)$, $z_2 = \left(\frac{1}{3},0 \right)$, $z_3 = \left(\frac{1}{2}, 0 \right)$, and $z_4 = \left(\frac{2}{3}, 0\right)$, then
\begin{equation}
    \G\left( \mathcal{D}_{\sqrt{3}}^{-1} h_3, \Gamma_N \times \Z \right), \ \Gamma_N = \bigcup_{k=1}^N (\Z+z_k) \quad \text{ is not a frame for } N=2,3,4.
\end{equation}

\medskip

$\bullet$ $\mathcal{Z}_{2} h_{4\ell+3} (z) = 0$, $z \in \left \lbrace(0,0), \left(0,\frac{1}{4} \right), \left(1, 0 \right), \left(\frac{1}{2}, 0 \right), \left(\frac{3}{2}, 0\right) \right \rbrace + D_{2}\Z^2$.

\noindent
The location of the zeros in $[0,2) \times [0,\frac{1}{2})$ is depicted in Figure~\ref{fig:zeros4}.
\begin{figure}[ht]
    \centering
    \includegraphics[width=0.75\linewidth]{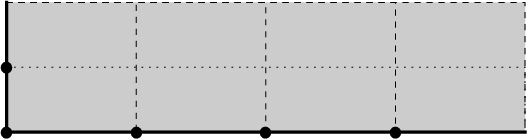}
    \caption{Known zeros for $\mathcal{Z}_2 h_{4\ell+3}$ in the fundamental cell of $2 \Z \times \frac{1}{2} \Z$.}
    \label{fig:zeros4}
\end{figure}

Note that the first three zeros in any of the above lists are due to the parity of $h_{4\ell+3}$. The cases for $\frac{1}{\sqrt{a}}$ are covered by an application of the Poisson summation formula, similarly to \eqref{eq:FT-Zak}, which results in simply rotating the rectangle in figure~\ref{fig:zeros2}, \ref{fig:zeros3}, and \ref{fig:zeros4} by 90 degrees.

Thus, specifically picking $h_3$ (i.e., $\ell=0$) as window function, we are able to design a wealth of Gabor systems over periodic configurations of density 5 which are not frames.

\section{Open questions and new zeros}\label{sec:questions}
While our results give new insights into the mysterious world of Gabor systems with Hermite functions, they immediately raise more questions.
\begin{enumerate}
    \item[(Q1)] For any $n \geq 1$, is there a $\Gamma$ with $D(\Gamma) > n+1$ such that $\G(h_n,\Gamma)$ is not a frame?
    \item[(Q2)] How many zeros does the Zak transform of $h_n$ have?
    \item[(Q3)] How do the zeros of $\mathcal{Z}_a h_n$ differ from the zeros of $\mathcal{Z}h_n$?
\end{enumerate}

In general, it seems that considering $\mathcal{Z}_a h_n$, $a \in \R_+\setminus\{1\}$, introduces new non-trivial zeros for $n \geq 2$, as shown in \cite{BooZakZuc83}, \cite{HorLemVid25}, \cite{Lem16}.

Conjectures about the frame property of Gabor systems with Hermite functions tend to be risky. However, we believe that the following statement about the Zak transform holds.
\begin{conjecture}
    For a generalized first Hermite function with parameters $a$ and $s$, i.e.,
    \begin{equation}
        \widetilde{h}_1(t) = \widetilde{h}_1(t;a,s) = \frac{1}{\sqrt{a}} h_1\left( \frac{t}{a} \right) e^{\pi i s t^2}, \quad a \in \R_+, \ s \in \R,
    \end{equation}
    the Zak transform $\mathcal{Z} \widetilde{h}_1(z)$ has exactly 3 zeros in $[0,1)^2$ and these are determined by \eqref{eq:Zak_zeros_odd}.
\end{conjecture}
As $h_1(t) = \widetilde{h}_1(t;1,0)$, the crucial point is that no new non-trivial zeros occur in the Zak transform when varying the parameters. This would be comparable to the Gaussian case. Note that Theorem~\ref{thm:frame_zeros}, Proposition~\ref{pro:frame_zeros}, and Corollary~\ref{cor:location_zeros} hold for generalized Hermite functions. Thus, Theorem~\ref{thm:h1} can be extended to periodic configurations $\Gamma(\Lambda, \{z_m\}_{m=1}^M)$, $\Lambda \subset \R^2$ a lattice with $D(\Lambda) = 1$, by using the results presented in \cite{Fau25_SampTA}.

For a dilated second Hermite function, we now prove the existence of a Gabor system with an oversampling rate of 5, which is not a frame. As shown in \cite{BooZakZuc83} or \cite[Lem.~3]{Lem16}, for $\omega = 1/2$, $\mathcal{Z} (\mathcal{D}^{-1}_a h_2)(x,1/2)$ possesses three zeros for $a \in \{\sqrt{2},\sqrt{3}\}$. We pick $a = \sqrt{2}$ and prove the existence of two more zeros: $z_0 = (x_0,0)$ and $\widetilde{z}_0 = (1-x_0,0)$ with $x_0 \in (0,1/2)$.

The second Hermite function is explicitly given by
\begin{equation}
    h_2(t) = 2^{-1/4} (-1 + 4 \pi t^2) e^{- \pi t^2}.
\end{equation}
The Zak transform $\mathcal{Z} \left(\mathcal{D}^{-1}_{\sqrt{2}} h_2 \right)(x,0)$ with $\omega = 0$ is a real-valued, continuous, even, periodic function on $\R$, depicted in Figure~\ref{fig:Zak_h2}. By using \eqref{eq:FT-Zak} and $\F \mathcal{D}^{-1}_a h_2 = - \mathcal{D}_a h_2$, we obtain
\begin{align}
    \mathcal{Z} \left(\mathcal{D}^{-1}_{\sqrt{2}} h_2 \right)(x,0)
    & = - \mathcal{Z} \left(\mathcal{D}_{\sqrt{2}} h_2 \right)(0,-x)\\
    & = 2^{-1/2} \sum_{k \in \Z} (1-2\pi k^2) e^{-\frac{\pi}{2} k^2} e^{-2 \pi i k x}\\
    & = 2^{-1/2} \sum_{k \in \Z} (1-2\pi k^2) e^{-\frac{\pi}{2} k^2} \cos(2 \pi k x).
\end{align}
\begin{figure}[ht]
    \centering
    \includegraphics[width=0.6\linewidth]{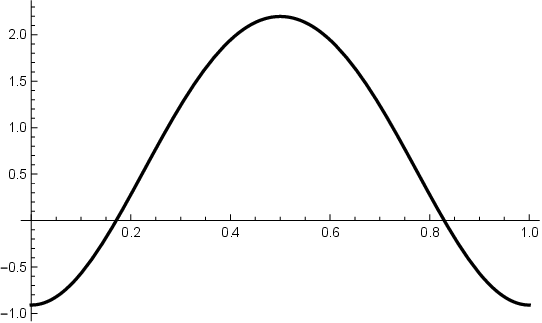}
    \caption{The Zak transform $\mathcal{Z}\left(\mathcal{D}^{-1}_{\sqrt{2}} h_2 \right)(x,0)$ on $[0,1]$.}
    \label{fig:Zak_h2}
\end{figure}

\noindent
We multiply by $2^{1/2}$ and evaluate at $x=0$ and $x = 1/2$. Note that $e^{-\frac{\pi}{2}} > 1/5$.
\begin{equation}
    (x=0): \quad 2^{1/2} \mathcal{Z} \left(\mathcal{D}^{-1}_{\sqrt{2}} h_2 \right)(0,0) = 1 + 2 \sum_{n \geq 1} \underbrace{(1- 2 \pi n^2)}_{<0} e^{-\frac{\pi}{2} n^2} < 1 + 2 (1- 2 \pi) e^{-\frac{\pi}{2}} < 0.
\end{equation}
\begin{align}
    (x=1/2): \quad
    2^{1/2} \mathcal{Z} \left(\mathcal{D}^{-1}_{\sqrt{2}} h_2 \right)(1/2,0)
    & = \sum_{k \in \Z} (-1)^k (1-2\pi k^2) e^{-\frac{\pi}{2} k^2}\\
    & > 1 + 2(-1 + 2\pi)e^{-\frac{\pi}{2}} - 2\sum_{n \geq 2} (1 + 2 \pi n^2) e^{-\frac{\pi}{2} n^2}\\
    & > 1 + 2(-1 + 2\pi)e^{-\frac{\pi}{2}} - 2\sum_{n \geq 2} (1 + 2 \pi n^2) e^{-\pi n}.
\end{align}
We derive closed-form expressions for the series in the \nameref{appendix}. The analytic expressions can be evaluated to arbitrary precision by hand. However, we used Mathematica \cite{Mathematica}, which can actually evaluate the series symbolically, to obtain the numerical bound
\begin{equation}
    2\sum_{n \geq 2} (1 + 2 \pi n^2) e^{-\pi n} < 0.11
    \quad \text{ while } \quad
    1 + 2(-1 + 2 \pi) e^{-\frac{\pi}{2}} > 3.
\end{equation}
Thus, it follows that $\mathcal{Z} \left(\mathcal{D}^{-1}_{\sqrt{2}} h_2 \right)(1/2,0) > 0$, which proves the existence of a zero $(x_0,0)$ of $\mathcal{Z}\left(\mathcal{D}^{-1}_{\sqrt{2}} h_2\right)(x,0)$ with $x_0 \in (0, 1/2)$. By symmetry, another zero $(1-x_0,0)$ exists. Thus, combined with the results from \cite{BooZakZuc83} or \cite{Lem16}, we know that $\mathcal{D}_{\sqrt{2}}^{-1} h_2$ has at least five distinct zeros in $[0,1)^2$. This guarantees the existence of a Gabor system with $\mathcal{D}_{\sqrt{2}}^{-1} h_2$ and a periodic structure of density 5, which is not a frame. We leave the exact location of the new zeros open. Numerically, they are approximately at $x_0=0.171$ and $\widetilde{x}_0=0.829$.

Lastly, we note that $\mathcal{Z}_{\sqrt{3}} h_3(z)$ (Figure~\ref{fig:zeros3}) possesses 2 additional zeros for $\omega = \frac{1}{2 \sqrt{3}}$ at approximately $x_0 = 0.256 \cdot \sqrt{3}$ and $\widetilde{x}_0 = 1 - x_0 = 0.744 \cdot \sqrt{3}$, as indicated in Figure~\ref{fig:Zak_h3}.
\begin{figure}[ht]
    \centering
    \includegraphics[width=0.6\linewidth]{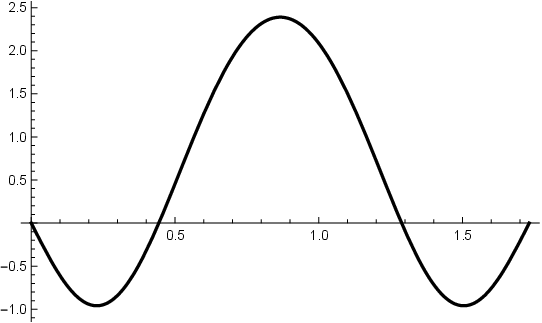}
    \caption{The dilated Zak transform $\mathcal{Z}_{\sqrt{3}} h_3 \left(x, \tfrac{1}{2 \sqrt{3}} \right)$ on $[0,\sqrt{3}]$. For $\ell \in \Z$, $\mathcal{Z}_{\sqrt{3}} h_3 \left(x+\ell, \tfrac{1}{2 \sqrt{3}} \right) = (-1)^\ell \mathcal{Z}_{\sqrt{3}} h_3 \left(x+\ell, \tfrac{1}{2 \sqrt{3}} \right)$ by the quasi-periodicity.}
    \label{fig:Zak_h3}
\end{figure}

We estimated the location using Mathematica \cite{Mathematica}. The proof of their existence is similar to that just above. Thus, there exists a periodic structure $\Gamma = \bigcup_{k=1}^7\left(D_{\sqrt{3}} \Z^2 + z_k\right)$ which has density 7 such that the Gabor system $\G(h_3,\Gamma)$ is not a frame.

% Based on what we have derived, it is tempting to formulate a last question.
% \begin{itemize}
%     \item[(Q4)] Does the Zak transform $\mathcal{Z}_{\sqrt{n}} h_n$ possess $2n+1$ zeros? 
% \end{itemize}

\section*{Appendix}\label{appendix}
    We use the following identities. For $|q| < 1$, we have
    \begin{align}
        f_N(q) = \sum_{n \geq N} q^n & = \frac{q^N}{1-q},\\
        q f_N'(q) = \sum_{n \geq N} n q^n & = \frac{N q^N}{1-q} + \frac{q^{N+1}}{(1-q)^2},\\
        q^2 f_N''(q) = \sum_{n \geq N} n^2 q^n
        & = \frac{N^2  q^N}{1-q} + \frac{N q^{N+1}}{(1-q)^2} + \frac{(N+1) q^{N+1}}{(1-q)^2} + 2 \frac{q^{N+2}}{(1-q)^3}\\
        & = q^N \frac{N^2(1-q)^2 +(2N+1)q(1-q)+2q^2}{(1-q)^3}\\
        & = q^N \frac{\left(N+(1-N)q\right)^2+q}{(1-q)^3}.
    \end{align}
    Picking specifically $N=2$ and $q = e^{-\pi}$, we obtain
    \begin{align}
        \sum_{n \geq 2} e^{-\pi n} & = \frac{e^{-2\pi}}{1-e^{-\pi}} \approx 0.00195179,\\
        \sum_{n \geq 2} n^2 e^{-\pi n} & = \frac{e^{-2 \pi } \left(\left(2-e^{-\pi}\right)^2 + e^{-\pi}\right)}{\left(1-e^{-\pi}\right)^3} \approx 0.0082588.
    \end{align}
    Together, this leads to the bound $2\sum_{n\geq 2}(1+2\pi n^2)e^{-\pi n} < 0.11$.

%\section*{Declarations}
%\subsection*{Conflict of interest}
%The author declares that there is no conflict of interest.

% \bibliographystyle{abbrv}
% \bibliography{mybib}

\end{document}